\documentclass{article}

\usepackage[english]{babel}
\usepackage[utf8]{inputenc}
\usepackage{amsmath}
\usepackage{amsthm}
\usepackage{amssymb}
\usepackage{graphicx}
\usepackage[colorinlistoftodos]{todonotes}
\usepackage{hyperref} 
\usepackage{color}
\usepackage{verbatim}

\newtheorem{theorem}{Theorem}
\newtheorem{lemma}[theorem]{Lemma}
\newtheorem{corollary}[theorem]{Corollary}

\theoremstyle{definition}

\title{Alternator Coins}
\author{Benjamin Chen, Ezra Erives, Leon Fan,\\
Michael Gerovitch, Jonathan Hsu, Tanya Khovanova,\\
Neil Malur, Ashwin Padaki, Nastia Polina,\\
Will Sun, Jacob Tan, Andrew The}
\date{}

\begin{document}

\maketitle

\begin{abstract}
We introduce a new type of coin: \textit{the alternator}. The alternator can pretend to be either a real or a fake coin (which is lighter than a real one). Each time it is put on a balance scale it switches between pretending to be either a real coin or a fake one. 

In this paper, we solve the following problem:
You are given $N$ coins that look identical, but one of them is the alternator. All real coins weigh the same. You have a balance scale which you can use to find the alternator. What is the smallest number of weighings that guarantees that you will find the alternator?
\end{abstract}

\section{Introduction}

Mathematicians have been fascinated with coin puzzles for a long time. The simplest coin puzzle is formulated like this:

\begin{quote}
You are given $N$ coins that look identical, but one of them is fake and is lighter than the other coins. All real coins weigh the same. You have a balance scale that you can use to find the fake coin. What is the smallest number of weighings that guarantees finding the fake coin?
\end{quote}

The above puzzle first appeared in 1945. Since then there have been many generalizations of this puzzle \cite{GN}. A new generalization that inspired this paper appeared in 2015 \cite{KKP}. This generalization introduces a new type of coin, called a \textit{chameleon coin}, which can mimic a fake coin or a real coin. The chameleon coin has a mind of its own and can choose how to behave at any weighing. It is impossible to find chameleon coins among real coins as the chameleons can pretend to be real all the time. An interesting question to ask is: given that a mix of $N$ identical coins contains one chameleon and one fake coin, find two coins one of which is guaranteed to be a fake \cite{KKP}. 

We can draw a parallel between coin puzzles and logic puzzles. Real coins are similar to truth-tellers, and fake coins are similar to liars. Many logic puzzles include \textit{normal} people: people who sometimes tell the truth and sometimes lie. Thus a chameleon coin is an analogue of a normal person. In addition to normal people, some logic puzzles have \textit{alternators}: people who alternate between telling the truth and lying. In logic puzzles if you are talking to one normal person s/he can behave consistently as a truth-teller or a liar, and it is impossible to find out who this person is. It is different with the alternating person. To identify them, you can just ask them how much is two plus two---twice.

Coming back to coins. As chameleon coins are analogues of normal people in logic puzzles, it would be natural to introduce the analogues of the alternators to coin puzzles. It would be logical to call such a coin \textit{the alternator}. The alternator can mimic a fake coin or a real coin. But there is a deterministic rule. The alternator switches the behavior each time it is on the scale. Unlike the chameleon, the alternator coin can always be found.

In this paper we solve the question of finding one alternator among $N$ coins using the balance scale in the minimum number of weighings. In Section~\ref{sec:alternator} we pose the problem, provide notation, small examples, and trivial bounds. In Section~\ref{sec:bound} we provide an information-theoretic argument for a stronger lower bound. In Section~\ref{sec:strategy} we produce a strategy that matches the lower bound, thus solving the problem.

This project was researched at the PRIMES STEP program, which is a younger branch of the MIT PRIMES program. The goal of the PRIMES program is to help gifted high-schoolers conduct research in mathematics. The PRIMES program started in 2011 and has been extremely successful \cite{EGK}. In 2015 the PRIMES staff decided to start a new program, PRIMES STEP, for middle schoolers. The goal is to train gifted students in middle school for math competitions, to teach them to think mathematically, and to do research with them. In the fall of 2015, the students worked on a logic project, and together with their mentor Tanya Khovanova wrote a paper, \textit{Who is Guilty?} \cite{guilty}. 

The Alternator Coin project was conducted at PRIMES STEP in the spring of 2016. The leader of the project was Tanya Khovanova. Her co-authors were in seventh or eighth grade at that time.

\section{The Alternator Coin}\label{sec:alternator}

We are given $N$ identically looking coins. All but one coin are real and weigh the same. One coin is special and is called \textit{the alternator coin}. It alternatively mimics a real coin and a fake coin. That means, when the alternator is put on a balance scale it either weights the same as a real coin or is lighter. The alternator is similar to the \textit{chameleon coin} first defined in \cite{KKP}. The chameleon coin can choose randomly and independently how to behave. Unlike the chameleon coin, the alternator is more deterministic. It switches its behavior each time it is put on the scale. As usual in coin puzzles, we have a balance scale and we need to find the alternator. More precisely, we need to find the smallest number of weighings that guarantees finding the alternator coin.

We denote the smallest number of weighings as $a(N)$.

In addition to $a(N)$ we study two more sequences. We can simplify our problem by assuming that the status of the alternator coin is known in advance. We call it the \textit{deterministic alternator}. Sequence $f(N)$ is the smallest number of weighings that guarantees to find the alternator among $N$ coins if the alternator starts as fake. Sequence $r(N)$ is the smallest number of weighings that guarantees to find the alternator among $N$ coins if the alternator starts as real.

We also want to introduce the \text{state} of the alternator. We say that the alternator is in $f$-state if the next time it will be on the scale it will behave as a fake coin. We say that the alternator is in $r$-state if the next time it will be on the scale it will behave as a real coin. That is to say, for calculating $f(N)$, correspondingly $r(N)$, we assume that the alternator starts in the $f$, correspondingly $r$, state. For calculating $a(N)$ we do not know the starting state of the alternator. We call it the $a$-state.

Note that after the alternator is found, even if it starts in the $a$-state, we can calculate its state during each weighing retroactively. There is one exception. It is possible to deduce that the alternator is the coin that was never on the scale. In this case, if the alternator starts in the $a$-state, we would still not know its state, when we find it.

\subsection{One fake coin}

Here we remind the readers the standard solution to the puzzle with one fake coin. We will use similar ideas with the alternator coin later.

Suppose there is a strategy that finds a fake coin in $w$ weighings. Suppose coin number $i$ is fake. Then there is a sequence of weighings after which we determine that the $i$-th coin is indeed the fake coin. The output of each weighing is one of three types:

\begin{itemize}
\item E---when the pans are equal weights.
\item L---when the left pan is lighter.
\item R---when the right pan is lighter.
\end{itemize}

We can represent the sequence of weighings that results in our conclusion that the $i$-th coin is fake as a string of three letters: E, L, and R. Obviously, the same string cannot correspond to two different coins. That means that the number of coins that can be processed in $w$ weighings is not more than $3^w$.

On the other hand, it is easy to produce a strategy that finds the fake coin out of $N$ coins in $\lceil \log_3 N \rceil$ weighings. For example, if the number of coins is $3^w$, we can divide all the coins into three parts with $3^{w-1}$ coins in each. We put two parts on the scale and if the scale unbalances, then the fake coin is on the lighter pan. If the scale balances, then the fake coin is in the pile that was not on the scale. This way with each weighing we make the pile containing the fake coin three times smaller. Using this algorithm we can find the fake coin in $w$ weighings. If the total number of coins is not a power of three, the same idea works. We will leave the details to the reader.

\subsection{Trivial bounds}

The alternator is trickier than just a regular fake coin, so we expect to use more weighings. Also the $a$ state provides us less information then the $f$ and $r$ states. That means, $a(N) \geq r(N)$ and $a(N) \geq f(N)$. In the following lemma, the solution for one fake coin allows us to set trivial lower and upper bounds for the alternator coin.

\begin{lemma} If the total number $N$ of coins is in the range: $3^{k-1} < N \leq 3^k$, then the bounds for the $a$ and $r$ starting states are: 
$k + 1 \leq a(N),\ r(N) \leq 2k.$ The bounds for the $f$ starting state are: $k \leq f(N) \leq 2k - 1.$
\end{lemma}

\begin{proof}
We proceed with the lower bound. The same information-theoretic argument as we presented above works for the $f$-state. In addition to that, if the alternator starts in the $r$-state, then the first weighing will balance. It will not provide any information for any coin; it will just change the state of the coins that are on the scale. As the $a$-state is not better than the $r$-state, the same lower bound works for the $a$-state too.

The upper bound is due to the following strategy. Do the same thing as if looking for the fake coin, but perform every weighing twice. If the alternator participates in two weighings in a row, it has to act as fake in one of the weighings. This way after two weighings we will know which of the three piles contains the alternator. If the alternator starts in the $f$-state, then we do not need to repeat the first weighing twice. 
\end{proof}

For example, this means that $r(2) = r(3)= a(2)=a(3)=2$ and $f(2)=f(3)=1$. Also, if $4 \leq N \leq 9$, then $a(N)$ and $r(N)$ are 3 or 4.

\subsection{Small examples}

For a small number of coins we searched all possible strategies. We present the weighing strategies in a way that will be useful for our induction later. The following properties of our strategies are important to notice:

\begin{itemize}
\item We present the strategies for the $a$ case.
\item The same strategy works for the $r$ case if we ignore the branch when the first weighing unbalances.
\item The same strategy without the first weighing works for the $f$ case.
\item The strategy for $2k+1$ coins is generated from the strategy for $2k$ coins: one coin is put aside from the start and if all the weighings balance, the alternator is that put-aside coin.
\end{itemize}

\subsubsection{Two or three coins}

There is only one possible type of weighing we can do: compare one coin to another coin on the scale. We can find the alternator coin in two weighings by comparing the first and the second coin twice. If one of these coins is the alternator, it will reveal itself. If not, which can only happen if the total number of coins is three not two, then after both weighings balance, we know that the alternator is coin number 3.

We  see that $f(2) = f(3) = 1$ and $r(2) = r(3)= a(2)=a(3)=2$.

\subsubsection{Four or five coins}

The trivial bound shows that we need at least 3 weighings for the $a$ and $r$ case and at least two weighings for the $f$ case. Here we show how to resolve the $a$ case in three weighings. The numbers are the indices of the coins.

Compare coins 1, 2 versus 3, 4. If the weighing is unbalanced, then the lighter pan contains the alternator. After that we can find the alternator in two weighings. If the weighing is balanced, then the second weighing is comparing 1 to 2, and the third weighing is comparing 3 to 4. At this time, if the alternator is among the first four coins, it will reveal itself. If all the weighings balance, then coin 5 is the alternator.

Therefore, $a(4) = r(4) = a(5) = r(5) = 3$ and $f(4) = f(5) =2$.

\subsubsection{Other strategies}

We presented the strategies that will be a part in our induction process. It is possible to have completely different strategies. Here is a different strategy for finding the alternator in three weighings when $N=4$.

Compare 1 to 2 in the first weighing. Compare 2 to 3 in the second weighing. Compare 1 to 3 in the third weighing. If at any point one of the weighings is unbalanced, then the lighter pan contains the alternator. If all the weighings balance, we see that every coin participated in two weighings each. This means that all these coins are real and the fourth coin is the alternator.

As you might notice for the examples we have $a(N) = r(N) = f(N)+1$. Is this always true? To maintain the suspense, let us hold back the answer to this question.

\section{A Better Lower Bound}\label{sec:bound}

To present and prove our new and better bound we first need to introduce the Jacobsthal numbers.

\subsection{Jacobsthal numbers}

The Jacobsthal numbers $J_n$ are defined as $J_n = (2^n - (-1)^n)/3$. This is sequence A001045 in the OEIS \cite{OEISJ}. The Jacobsthal numbers satisfy the following recursion: $J_{n+1} = J_n + 2J_{n-1}$. Indeed, the two geometric series $2^n$ and $(-1)^n$ satisfy this recursion, together with any linear combination of them.

The Jacobsthal numbers are the coolest numbers you never heard about. The OEIS page has a lot of different definitions for the Jacobsthal numbers. 

For example, $J_n$ is the number of ways to tile a 3-by-$(n-1)$ rectangle with 1-by-1 and 2-by-2 square tiles. Also $J_n$ is the number of ways to tile a 2-by-$(n-1)$ rectangle with 1-by-2 dominoes and 2-by-2 squares. We leave it to the reader to prove these properties.

Another property that we also leave to the reader: the product of two successive Jacobsthal numbers is always a triangular number.

But we digress. We do not need these cool properties for our future progress. What we need is the following lemma which is also easy to prove:

\begin{lemma}\label{lem:2j}
$J_n = 2J_{n-1} - (-1)^n$.
\end{lemma}
 	
We actually do not need this lemma but rather its consequence.

\begin{corollary}\label{cor:Jsplit}
If a number $k$ is between two successive Jacobsthal numbers: $J_n < k \leq J_{n+1}$, then $k = 2J_{n-1} + m$, where $0 \leq m \leq J_n$.
\end{corollary}

The Jacobsthal numbers are very important in this paper. We will see that each of our three sequences increase by 1 right after every $N$ that is a Jacobsthal number.

\subsection{New bound}

We use information theory again to produce a better bound.

\begin{theorem}
The number of coins we can process in $w$ weighings is not more than $J_{w+2}$ if the alternator is in $f$-state. In $r$- or $a$-state the number of coins we can process in $w$ weighings is not more than $J_{w+1}$.
\end{theorem} 

\begin{proof}
Suppose there is a strategy. Let us assign a string in the alphabet ELR to each coin. If coin $i$ is the alternator, the string corresponds to the strategy that finds this coin. Strings assigned to different coins must be different. The length of the string is not more than the number of weighings. The new and important observation here is the following: letters L and R cannot follow each other. They must be separated by at least one letter E. Suppose our alternator coin is on one of the pans that is lighter. If the alternator does not participate in the next weighing, then the weighing will balance. If it does participate in the next weighing, the weighing will balance too as the alternator will be in the $r$-state.

Now we need to calculate the maximum number of strings of given length $w$ with this property. Notice that it is theoretically possible to have the strings of shorter length to point to a coin, that is, some coins might be found faster than others. But if a shorter string points to a coin, then all the strings with the same prefix point to the same coin. Thus to find the theoretical maximum we should only count strings of fixed length $w$. 

Let us calculate the number of such strings by induction. Denote this number as $s(n)$. We have 1 string (an empty one) of length zero and 3 strings of length one. The number of strings of length $k$ can be calculated is follows. If the string starts with E, then it can be followed by any such string of length $k-1$. If it starts with L or R, it must be followed by E and then any such string of length $k-2$. Therefore, we have a recursion: $s(n) = s(n-1) + 2s(n-2)$. This is the same recursion as for the Jacobsthal numbers. The initial terms are the Jacobsthal numbers shifted by 2. So $s(n) = J_{n+2}$. What remains to note is that if the alternator is in $r$-state the first letter of the string must be E.
\end{proof}

The lower bound for the number of weighings increases after each Jacobsthal number. In the next section we will see that the bound is precise.

\section{Weighing Strategy}\label{sec:strategy}

Here we want to produce a strategy that can find the alternator in the same number of weighings provided by the bound above. First we deal with an $f$ and $r$ cases and the total of $N$ coins. We define the strategy recursively.

\textbf{Strategy}

\begin{itemize}
\item Suppose the coins are in the $f$-state. Suppose $N$ is more than $J_k$ and not more than $J_{k+1}$. We weigh two piles each containing $J_{k-1}$ coins. If the scale unbalances, then the alternator is in one of the piles on the scale and it will switch its state to $r$. Now we need to process $J_{k-1}$ coins in state $r$. If the scale balances, the alternator is not on the scale, and is still in state $f$. It is among other coins and did not switch its state. Now we need to process $N-2J_{k-1}$ in state $f$.
\item Suppose the coins are in the $r$-state and $N$ is even. We arrange two piles of $N/2$ coins on the scale. After the weighing balances, all the coins will switch their state to $f$. Now we need to process $N$ coins in state $f$. 
\item Suppose the coins are in the $r$-state and $N$ is odd. We put one coin aside and proceed with an even number, $N-1$, of coins as above. At the end of the weighings, if we do not find the alternator, then the put-aside coin is the alternator.
\end{itemize}

Notice that the strategy matches our examples for the total of 2 to 5 coins.

In the following theorem we assume that the number of weighings is at least 1. We use the strategy above to prove the theorem.

\begin{theorem}\label{thm:rf}
For the $f$-state, the number of coins $N$ we can process in $w$ weighings is $J_{w+1} < N \leq J_{w+2}$. For the $r$-state, the number of coins $N$ we can process in $w$ weighings is $J_{w} < N \leq J_{w+1}$.
\end{theorem} 

\begin{proof}
We proceed by induction on $k$, the index of the Jacobsthal numbers. Assuming the theorem holds for the number of coins up to $J_k$, we will show that it holds for the number of coins up to $J_{k+1}$. 

For the base of induction we use the small examples we have already covered. For the $f$-state, in one weighing we can process $N$ coins, where $J_{2}=1 < N \leq J_{3} =3$. In two weighings we can process $N$ coins, where $J_{3}=3 < N \leq J_{4} =5$. For the $r$-state, in one weighing we can process $N$ coins, where $J_{1}=1 < N \leq J_{2} =1$; that is, nothing can be done in one weighing. In two weighings we can process $N$ coins, where $J_{2}=1 < N \leq J_{3} =3$. In three weighings we can process $N$ coins, where $J_{3}=3 < N \leq J_{4} =3$. For the number of coins up to $5 = J_4$ our theorem holds. 

Suppose the theorem is true, and we can find a strategy for the number of coins up to $J_k$, where $k > 4$ in at least $k-2$ weighings for the $f$-state and at least $k-1$ weighings for the $r$-state. Consider the number of coins $N$ that is more than $J_k$ and not more than $J_{k+1}$. For our induction step we need to find the strategy for the $f$-state in $k-1$ weighings and for the $r$-state in $k$ weighings.

First consider the $f$ state. By Corollary~\ref{cor:Jsplit}, we have $N=2J_{k-1} + m$, where $0 \leq m \leq J_k$. Our strategy is to weigh two piles each containing $J_{k-1}$ coins and have $m$ coins left outside the scale. If the scale unbalances, then the alternator is in one of the piles on the scale. It switches state to $r$. By our induction hypothesis $r(J_{k-1}) = k-2$. If the scale balances, then the alternator is in the leftover pile and has state $f$. By our induction hypothesis $f(m) \leq f(J_k) = k-2$. Given that we used one weighing, the total number of weighings that we need is $k-1$.

Next consider the $r$-state. If $N$ is even, we weigh all the coins switching their state to $f$. By the induction hypothesis, we just showed how to find the alternator among these coins in $k-1$ more weighings. Therefore, we can find the alternator in $k$ weighings total.

Suppose $N$ is odd. We put aside one coin and proceed with the rest of the coins as above. We need to show that we will not end up with a contradiction. That is, if all the weighings balance, the only possibility for the alternator is to be the coin that is put aside. Note that if one of the weighings is unbalanced, then the coin that we put aside cannot be the alternator. 

Suppose $N$ is even and all the weighings balance. In the first weighing we have all the coins on the scale and they turn their state to $f$. If the second weighing balances, that means that all the coins on the scale are real, so they do not participate in any further weighings. The number of leftover coins is even and they are in state $f$, so we proceed with them. After each successive balanced weighing, we remove an even number of coins from consideration and we are left with an even number of coins. At the end, the number of coins that are left must be zero. Indeed, if we are left with more than 1 coin that hasn't been on the scale for the second time, we cannot differentiate which one is the alternator. That means the alternator must be on the scale at least twice. Therefore, there is at least one unbalanced weighing.

Now let us go back to odd $N$. If all the weighings balance, the alternator must be the put-aside coin.
\end{proof}

We now want to consider the $a$ case, when the state of the alternator is not known. Here is the strategy:

\textbf{Strategy}

\begin{itemize}
\item If $N$ is even we arrange two piles of $N/2$ coins on the scale. If the weighing balances, all the coins on the scale will be in state $f$. Now we need to process $N$ coins in state $f$, so we follow the strategy above. If the weighing unbalances, the alternator is on the lighter pan, and in state $r$. We use the strategy above to process these coins.
\item If $N$ is odd, we put one coin aside and proceed with the rest of the coins as above. At the end of the weighings, if we do not find the alternator, then the put-aside coin is the alternator.
\end{itemize}

In the following theorem we assume that the number of weighings is at least 1. We use the previous strategy to prove the theorem.

\begin{theorem}
For the $a$-state, the number of coins $N$ we can process in $w$ weighings is $J_{w} < N \leq J_{w+1}$.
\end{theorem} 

\begin{proof}
In the first weighing we weigh all or all but one coin. If the weighing unbalances, then the alternator is one of the $\lfloor N/2 \rfloor$ coins on the scale. By Lemma~\ref{lem:2j}, we see that $N/2 \leq J_w - (-1)^{w+1}/2$. Therefore, $\lfloor N/2 \rfloor  \leq J_w$. By Theorem~\ref{thm:rf} we can process not more than $J_w$ coins in $w-2$ weighings. That means we can process all the coins in $w-1$ weighings.

If the weighing balances, we see that all the coins that were on the scale must be in the $f$ state, and we can process them in $w-1$ weighings. The total number of weighings is $w$.

By similar reasoning as above, the coin that was put aside is the alternator if all the weighings balance.
\end{proof}

\begin{corollary}
$a(N) = r(N) = f(N) +1$.
\end{corollary}

\section{Acknowledgements}
We would like to thank the PRIMES STEP program for the opportunity to do this research. In addition, we are grateful to PRIMES STEP Director, Dr.~Slava Gerovitch, for his help and support.

\end{document}